\documentclass[11pt]{article}

\usepackage[a4paper,margin=1.12in]{geometry}
\usepackage{amsmath,amsthm,amssymb,mathtools}
\usepackage{xcolor}
\usepackage{microtype}
\usepackage{enumitem}
\usepackage[hidelinks]{hyperref}
\usepackage{cleveref}

\allowdisplaybreaks

\newtheorem{maintheorem}{Theorem}

\newtheorem{proposition}{Proposition}[section]
\newtheorem{lemma}[proposition]{Lemma}
\newtheorem{corollary}[proposition]{Corollary}
\theoremstyle{remark}
\newtheorem{remark}[proposition]{Remark}

\newcommand{\K}{\mathbb K}
\newcommand{\R}{\mathbb R}
\newcommand{\C}{\mathbb C}
\newcommand{\N}{\mathbb N}
\newcommand{\G}{\Gamma}
\newcommand{\eps}{\varepsilon}
\newcommand{\norm}[1]{\left\lVert #1\right\rVert}
\newcommand{\abs}[1]{\left|#1\right|}
\newcommand{\pair}[2]{\left\langle #1,#2\right\rangle}

\title{Effective Dimensions in Grothendieck--H\"older Inequalities}

\author{%
Joedson Santos\thanks{Departamento de Matem\'atica,
Centro de Ci\^encias Exatas e da Natureza, Universidade Federal da
Para\'iba, Jo\~ao Pessoa, PB, Brazil.
Email: \texttt{joedson.santos@academico.ufpb.br}}
\and
Diana Serrano-Rodr\'iguez\thanks{Departamento de Matem\'aticas,
Facultad de Ciencias, Universidad Nacional de Colombia, Sede Bogot\'a,
Bogot\'a, Colombia. Email: \texttt{diserranor@unal.edu.co}}
}

\date{}

\begin{document}
\maketitle

\begin{abstract}
What becomes of Grothendieck's inequality when its Hilbert-space
vectors are replaced by vectors in $\ell_p^r$ and $\ell_{p'}^r$?
For an $m\times n$ matrix, we show that the universal dimensional
loss is governed not by $r$ alone, but by the effective dimension
\[
d=\min\{r,m,n\}.
\]
More precisely, the optimal constant is bounded above and below,
up to dimension-free factors, by
\[
d^{|1/p-1/2|}.
\]
The upper bound rests on a pointwise Hilbertian stabilization:
for every matrix $A$,
\[
G_2^{(r)}(A)=G_2^{(d)}(A).
\]
Combining this stabilization with Lewis's Euclidean distortion
estimate gives the effective-dimensional Grothendieck--H\"older
bound.

Beyond the universal power law, we retain the finer rectangular
geometry. At $p=\infty$ we prove the exact compression formula
\[
\Gamma_\infty^\K(r;m,n)
=
\rho\!\left(
\ell_1^m(\K),
\ell_1^{\min\{r,n\}}(\K)
\right),
\]
and derive from it endpoint-sensitive lower bounds throughout the
full H\"older scale. Over the complex field, interpolation between
the Hilbertian problem and the exact rectangular endpoints yields
corresponding upper bounds. Thus the effective dimension determines
the universal exponent, whereas the tensor-norm endpoint retains
additional aspect-ratio information.

We determine the boundedness threshold when $p$ approaches $2$:
it is governed by
\[
|1/p-1/2|\log d.
\]
For complex two-row matrices, the endpoint tensor ratios are explicit,
and the general rectangular estimates give quantitative two-sided
bounds for every $1\le p\le\infty$.
\end{abstract}

\section{Introduction}

Let $A=(a_{ij})\in\K^{m\times n}$, where $\K=\R$ or $\C$.
A finite-dimensional form of Grothendieck's inequality says that
\begin{equation}\label{eq:intro-Grothendieck}
\sup_{\substack{x_1,\ldots,x_m\in B_{\ell_2^r}\\
                y_1,\ldots,y_n\in B_{\ell_2^r}}}
\abs{\sum_{i=1}^m\sum_{j=1}^n a_{ij}\pair{x_i}{y_j}}
\le
K_G^{\K}
\sup_{\substack{|s_i|\le1\\|t_j|\le1}}
\abs{\sum_{i=1}^m\sum_{j=1}^n a_{ij}s_it_j},
\end{equation}
where $K_G^{\K}$ does not depend on $r$, $m$, $n$, or $A$;
see \cite{Grothendieck1953,LindenstraussPelczynski1968}.
Thus allowing Hilbert-space vectors in place of scalars changes
the supremum by at most a universal factor.

Here we ask what happens if the two families of Hilbert-space
vectors are replaced by vectors in the dual pair
$\ell_p^r$ and $\ell_{p'}^r$. At $p=2$ we recover
\eqref{eq:intro-Grothendieck}. Away from $2$, a loss depending on
the dimensions is unavoidable. The question is which dimensions
are responsible for that loss.

There is a substantial theory behind this question.
Grothendieck-type inequalities for operators on $L_p$ spaces
and their tensor-product formulations appear, for example, in
\cite{CarlDefant1992,DefantFloret1993}. Defant and Junge studied
finite Marcinkiewicz--Zygmund inequalities for general operators
and the growth of their constants outside the dimension-free
range \cite{DefantJunge1997}. Their later vector-valued
Grothendieck inequality concerns an operator-space extension
different from the one considered here \cite{DefantJunge1999}.

The formulation closest to ours is that of Friedland, Lim and
Zhang. They express Grothendieck's constant as a norm of the
matrix multiplication tensor $\mu_{r,m,n}$ and study the
associated family of tensor norms
\cite{FriedlandLimZhang2019}. In the notation used below, their
estimate along the line of exponents considered here gives
\[
\norm{\mu_{r,m,n}}_{1,p,\infty}
\le K_G^{\K}r^{|1/p-1/2|}.
\]
In particular, the exponent $|1/p-1/2|$ and the special role
of $p=2$ are not new. What this estimate does not distinguish
is the effect of the two matrix dimensions: its right-hand side
still grows with $r$ when $m$ or $n$ is small.

The contribution of the present paper is therefore not the exponent
$|1/p-1/2|$ itself. There are two further questions which are invisible
in the estimate above.

The first is an effective-dimension question. If the vector dimension
$r$ is much larger than one of the matrix dimensions, there is no
reason a priori for all $r$ vector coordinates to remain relevant.
We prove that the correct universal dimensional parameter is
\[
d=\min\{r,m,n\}.
\]
At the Hilbertian point this comes from an exact pointwise
stabilization, valid separately for every matrix:
\[
G_2^{(r)}(A)=G_2^{(d)}(A).
\]
Away from $p=2$, the same effective dimension enters through the
Euclidean distortion of the spans generated by the two vector
families.

The second question is finer. The number $d$ determines the universal
power of the constants, but it does not contain all rectangular
information. At the endpoints $p=1,\infty$, the optimal constants
compress exactly to projective-to-injective tensor norm ratios whose
two dimensions need not coincide. We show that these exact endpoint
ratios generate lower envelopes throughout the full H\"older scale.
Over the complex field they also generate, by interpolation with the
Hilbertian problem, endpoint-sensitive upper envelopes. Thus the paper
separates two levels of geometry:
\[
\boxed{
\text{effective dimension }d
\quad\text{controls the universal power law},
}
\]
whereas
\[
\boxed{
\text{rectangular tensor ratios}
\quad\text{retain the finer aspect-ratio dependence}.
}
\]

We work over $\K\in\{\R,\C\}$. For $1\le p\le\infty$, let $p'$ denote the H\"older conjugate. Throughout, we use the convention $1/\infty=0$. If $A=(a_{ij})\in\K^{m\times n}$, set
\begin{equation}\label{eq:matrix-norm}
 \norm{A}_{\infty,1}:=
 \sup_{\norm{z}_\infty\le1}\norm{Az}_1
 =\sup_{\substack{\norm{s}_\infty\le1\\\norm{t}_\infty\le1}}
 \abs{\sum_{i=1}^m\sum_{j=1}^n a_{ij}s_it_j}.
\end{equation}
For $r\in\N$ define
\begin{equation}\label{eq:G-rmn}
 G_p^{(r)}(A):=
 \sup_{\substack{x_1,\ldots,x_m\in B_{\ell_p^r}\\
                  y_1,\ldots,y_n\in B_{\ell_{p'}^r}}}
 \abs{\sum_{i=1}^m\sum_{j=1}^n a_{ij}\pair{x_i}{y_j}},
\end{equation}
where $\pair{x}{y}=\sum_{k=1}^r x_ky_k$. \textcolor{black}{Over $\C$, this denotes the bilinear duality pairing, not the Hermitian inner product.} We then put
\begin{equation}\label{eq:Gamma-rmn}
 \G_p^{\K}(r;m,n):=
 \sup_{A\in\K^{m\times n}\setminus\{0\}}
 \frac{G_p^{(r)}(A)}{\norm{A}_{\infty,1}}.
\end{equation}
With the notation of \cite{FriedlandLimZhang2019},
\[
 \G_p^{\K}(r;m,n)=\norm{\mu_{r,m,n}}_{1,p,\infty}.
\]
We omit the superscript $\K$ when the scalar field is fixed.

For Banach spaces $X$ and $Y$ over $\K$, let
\[
\rho(X,Y):=
\sup_{0\ne z\in X\otimes Y}
\frac{\norm z_\pi}{\norm z_\varepsilon}
\]
denote the projective-to-injective tensor norm ratio.

For integers $u,v\ge1$, it will be convenient to write
\begin{equation}\label{eq:Ruv-def}
\mathcal R_{\K}(u,v)
:=
\rho\bigl(\ell_1^u(\K),\ell_1^v(\K)\bigr).
\end{equation}
The tensor flip gives
\[
\mathcal R_{\K}(u,v)=\mathcal R_{\K}(v,u).
\]
Moreover, zero padding shows that $\mathcal R_{\K}(u,v)$ is
nondecreasing in each variable. The exact endpoint compression proved
below can then be written simply as
\[
\Gamma_\infty^\K(r;m,n)
=
\mathcal R_{\K}\bigl(m,\min\{r,n\}\bigr).
\]

Our first answer is that the universal dimensional scale is governed by
\[
d=\min\{r,m,n\}.
\]
For each individual matrix $A$, the Hilbertian numerator already
stabilizes at this dimension:
\[
G_2^{(r)}(A)=G_2^{(d)}(A).
\]
Moreover,
\[
G_p^{(r)}(A)
\le
d^{|1/p-1/2|}G_2^{(d)}(A).
\]
Together with the endpoint constructions developed below, this gives
\[
\Gamma_p^\K(r;m,n)
\asymp
d^{|1/p-1/2|}
\]
with constants independent of all three dimensions and of $p$.

The effective dimension, however, determines only the universal power
law. A second level of the problem retains genuinely rectangular
information. At $p=\infty$ we prove the exact identity
\[
\Gamma_\infty^\K(r;m,n)
=
\mathcal R_{\K}\bigl(m,\min\{r,n\}\bigr),
\]
whereas at $p=1$ duality gives
\[
\Gamma_1^\K(r;m,n)
=
\mathcal R_{\K}\bigl(n,\min\{r,m\}\bigr).
\]
These endpoint quantities remember the aspect ratio even when the
effective dimension $d$ is fixed.

More generally, Theorem~\ref{thm:C} shows that every intermediate
endpoint dimension supplies a lower bound. For $p\ge2$,
\[
\Gamma_p^\K(r;m,n)
\ge
\sup_{1\le k\le\min\{r,n\}}
k^{-1/p}\mathcal R_{\K}(m,k),
\]
and there is a dual formula for $p\le2$. Over the complex field, the
exact rectangular endpoint also interpolates with the Hilbertian
constant to give an upper bound throughout the interior of the
H\"older scale.

Thus there are two complementary notions of dimension in the problem:
\[
\boxed{
d=\min\{r,m,n\}
\text{ controls the leading power},
}
\]
while
\[
\boxed{
\mathcal R_{\K}(m,k)
\text{ and }
\mathcal R_{\K}(n,k)
\text{ control finer rectangular effects}.
}
\]

The distinction can be substantial. For instance, if
$r=N^{10}$ and $m=n=N$, a bound expressed in terms of $r^{1/2}$
has size $N^5$ at $p=\infty$, whereas the universal effective scale is
only $N^{1/2}$. On the other hand, even after $d$ is fixed, the exact
endpoint may still vary with the aspect ratio. The $2\times3$ example
below shows that Hilbertian stabilization does not imply pointwise
endpoint stabilization.

The proof of the upper bound follows this observation closely.
The vectors from one family span a space of dimension at most
$\min\{r,m\}$. On that space, each vector from the other family
acts as a bounded functional. Lewis's estimate allows us to
replace this small space by a Euclidean one at a cost controlled
by its dimension \cite{Lewis1978}. Interchanging the two
families brings $n$ into the estimate.

For the lower bound, Fourier matrices give the required factor
in every complex dimension. In the real case, we use corners
of Hadamard matrices. Bounds on gaps between Hadamard orders
then give a lower factor of the form $1-O(d^{-5/6})$
\cite{Livinskyi2012,BrentOsbornSmith2016}.

Theorem \ref{thm:D} describes what happens when $p$ varies with the dimensions. If $d_N\to\infty$, the constants
remain bounded precisely when
\[
\abs{\frac1{p_N}-\frac12}\log d_N
\]
remains bounded. Thus the transition takes place when
$|1/p_N-1/2|$ is of order $1/\log d_N$. Inside that window,
we obtain bounds for the possible limiting constants rather
than asserting a single limit.

At $p=\infty$, the optimal rectangular constant has a more
precise description than the general effective-dimension estimate:
\[
\G_\infty^\K(r;m,n)
=
\rho\bigl(\ell_1^m(\K),
          \ell_1^{\min\{r,n\}}(\K)\bigr).
\]
This follows from the partition formula for the numerator.
In the complex two-row case it yields
\[
\G_\infty^\C(r;2,n)
=
\min\{r,n\}
\sin\!\left(\frac{\pi}{2\min\{r,n\}}\right).
\]
The compression concerns optimal constants; it does not imply
pointwise stabilization of $G_\infty^{(r)}(A)$ at
$\min\{r,m,n\}$. A $2\times3$ example below distinguishes
these statements.

For $d\in\N$, define the Hilbert-dimension-$d$ Grothendieck
constant by
\begin{equation}\label{eq:kappaGd}
\kappa_G^{\K}(d)
:=
\sup_{m,n\ge1}\Gamma_2^{\K}(d;m,n).
\end{equation}
Then $\kappa_G^{\K}(d)$ is nondecreasing and
\begin{equation}\label{eq:kappaGd-limit}
\kappa_G^{\K}(d)\uparrow K_G^{\K}.
\end{equation}
Finally, define the normalized square endpoint
\begin{equation}\label{eq:Bd}
 B_d^{\K}:=
 \frac{\G_\infty^{\K}(d;d,d)}{\sqrt d}.
\end{equation}

Our main result is the following.

\begin{maintheorem}\label{thm:A}
Let $r,m,n\in\N$, $1\le p\le\infty$, and put
\[
 d:=\min\{r,m,n\},
 \qquad
 a(p):=\abs{\frac1p-\frac12}.
\]
Then
\begin{equation}\label{eq:main-bottleneck}
 B_d^{\K}d^{a(p)}
 \le
 \G_p^{\K}(r;m,n)
 \le
 d^{a(p)}\G_2^{\K}(r;m,n)
 \le
 \kappa_G^{\K}(d)d^{a(p)}.
\end{equation}
More precisely, for every $A\in\K^{m\times n}$,
\begin{equation}\label{eq:Hilbert-saturation}
 G_2^{(r)}(A)=G_2^{(d)}(A),
\end{equation}
and
\begin{equation}\label{eq:pointwise-bottleneck}
 G_p^{(r)}(A)
 \le
 d^{a(p)}G_2^{(d)}(A)
 =d^{a(p)}G_2^{(r)}(A).
\end{equation}
\end{maintheorem}

The endpoint in Theorem \ref{thm:A} has the following deterministic lower estimate. Let $\mathcal H$ be the set of orders of real Hadamard matrices and define
\begin{equation}\label{eq:hplus}
 h_+(d):=\min\{h\in\mathcal H:h\ge d\}.
\end{equation}
The minimum exists because the Sylvester construction provides Hadamard matrices of every order $2^k$. Put
\begin{equation}\label{eq:beta-d}
 \beta_d^{\C}:=1,
 \qquad
 \beta_d^{\R}:=\sqrt{\frac{d}{h_+(d)}}.
\end{equation}

\begin{maintheorem}\label{thm:B}
For every $d\in\N$,
\begin{equation}\label{eq:Bd-lower}
 B_d^{\K}\ge\beta_d^{\K}.
\end{equation}
Moreover,
\begin{equation}\label{eq:beta-rate}
 \beta_d^{\C}=1,
 \qquad
 1-\beta_d^{\R}=O(d^{-5/6}).
\end{equation}
Therefore, for all $r,m,n\in\N$ and all $1\le p\le\infty$, with $d=\min\{r,m,n\}$,
\begin{equation}\label{eq:uniform-rectangular}
 \beta_d^{\K}d^{a(p)}
 \le
 \G_p^{\K}(r;m,n)
 \le
 \kappa_G^{\K}(d)d^{a(p)}.
\end{equation}
If $\K=\C$, the lower factor in \eqref{eq:uniform-rectangular} is exactly $1$ for every $d$. If $\K=\R$ and $d$ is a Hadamard order, it is also exactly $1$. In particular,
\begin{equation}\label{eq:rectangular-order}
 \G_p^{\K}(r;m,n)
 \asymp
 \min\{r,m,n\}^{|1/p-1/2|}
\end{equation}
with constants independent of $r,m,n$ and $p$.
\end{maintheorem}

The effective-dimension estimate can be refined by retaining the exact
rectangular endpoints. Put
\[
k_+:=\min\{r,n\},
\qquad
k_-:=\min\{r,m\}.
\]
The next result gives lower envelopes over all intermediate endpoint
dimensions. Over the complex field, the corresponding endpoint values
also give upper envelopes by interpolation.

\begin{maintheorem}[Rectangular endpoint envelopes]\label{thm:C}
Let $r,m,n\in\N$.

\medskip
\noindent
\textup{(i)} For every $1\le p\le\infty$, over either scalar field,
\begin{equation}\label{eq:rectangular-lower-plus}
\Gamma_p^\K(r;m,n)
\ge
\sup_{1\le k\le k_+}
k^{-1/p}\mathcal R_{\K}(m,k).
\end{equation}
Equivalently,
\begin{equation}\label{eq:rectangular-lower-plus-normalized}
\Gamma_p^\K(r;m,n)
\ge
\sup_{1\le k\le k_+}
\frac{\mathcal R_{\K}(m,k)}{\sqrt{k}}\,
k^{\,1/2-1/p}.
\end{equation}

\medskip
\noindent
\textup{(ii)} For every $1\le p\le\infty$, over either scalar field,
\begin{equation}\label{eq:rectangular-lower-minus}
\Gamma_p^\K(r;m,n)
\ge
\sup_{1\le k\le k_-}
k^{-1/p'}\mathcal R_{\K}(n,k),
\end{equation}
or equivalently
\begin{equation}\label{eq:rectangular-lower-minus-normalized}
\Gamma_p^\K(r;m,n)
\ge
\sup_{1\le k\le k_-}
\frac{\mathcal R_{\K}(n,k)}{\sqrt{k}}\,
k^{\,1/p-1/2}.
\end{equation}
Consequently, for every $1\le p\le\infty$,
\begin{equation}\label{eq:rectangular-lower-envelope}
\Gamma_p^\K(r;m,n)
\ge
\max\left\{
\sup_{1\le k\le k_+}k^{-1/p}\mathcal R_{\K}(m,k),
\sup_{1\le k\le k_-}k^{-1/p'}\mathcal R_{\K}(n,k)
\right\}.
\end{equation}

\medskip
\noindent
\textup{(iii)} Suppose that $\K=\C$. If $2\le p\le\infty$, then
\begin{equation}\label{eq:rectangular-upper-plus}
\Gamma_p^\C(r;m,n)
\le
\Gamma_2^\C(r;m,n)^{2/p}
\,
\mathcal R_{\C}(m,k_+)^{\,1-2/p}.
\end{equation}
Equivalently,
\begin{equation}\label{eq:rectangular-upper-plus-ratio}
\Gamma_p^\C(r;m,n)
\le
\Gamma_2^\C(r;m,n)
\left(
\frac{\mathcal R_{\C}(m,k_+)}
     {\Gamma_2^\C(r;m,n)}
\right)^{2a(p)}.
\end{equation}

\medskip
\noindent
\textup{(iv)} Suppose that $\K=\C$. If $1\le p\le2$, then
\begin{equation}\label{eq:rectangular-upper-minus}
\Gamma_p^\C(r;m,n)
\le
\mathcal R_{\C}(n,k_-)^{\,2/p-1}
\,
\Gamma_2^\C(r;m,n)^{\,2-2/p},
\end{equation}
or, equivalently,
\begin{equation}\label{eq:rectangular-upper-minus-ratio}
\Gamma_p^\C(r;m,n)
\le
\Gamma_2^\C(r;m,n)
\left(
\frac{\mathcal R_{\C}(n,k_-)}
     {\Gamma_2^\C(r;m,n)}
\right)^{2a(p)}.
\end{equation}
\end{maintheorem}

The lower bounds in Theorem~\ref{thm:C} \textcolor{black}{refine} the square
lower bound used in Theorem~\ref{thm:A}. For $p\ge2$, the latter is
recovered from \eqref{eq:rectangular-lower-plus} by choosing
$k=d=\min\{r,m,n\}$; for $p\le2$, it is recovered from
\eqref{eq:rectangular-lower-minus} with the same choice. The stronger
form \eqref{eq:rectangular-lower-envelope} shows that both rectangular
endpoint mechanisms are available simultaneously for every $p$. The
optimizing intermediate dimension need not equal $d$.

Similarly, the complex upper bounds retain the exact rectangular
endpoint rather than replacing it immediately by its universal
effective-dimensional estimate. Since Theorem~\ref{thm:A} at
$p=\infty$ gives
\[
\mathcal R_{\C}(m,k_+)
=
\Gamma_\infty^\C(r;m,n)
\le
\sqrt d\,\Gamma_2^\C(r;m,n),
\]
formula \eqref{eq:rectangular-upper-plus} always implies the upper
bound of Theorem~\ref{thm:A}; \textcolor{black}{for $2<p\le\infty$, it is strictly sharper whenever the
rectangular endpoint is smaller than the universal estimate.}
\textcolor{black}{The dual statement holds for $1\le p<2$.}

As a consequence of Theorems~\ref{thm:A} and \ref{thm:B}, the
fixed-$p$ dimensional exponent is immediate.

\begin{corollary}[Fixed-$p$ asymptotics]
\label{cor:fixed-p-asymptotic}
Fix $1\le p\le\infty$ with $p\ne2$. If
\[
d=\min\{r,m,n\}\longrightarrow\infty,
\]
then
\begin{equation}\label{eq:visual-asymptotic}
\Gamma_p^\K(r;m,n)
=
d^{\,|1/p-1/2|+o(1)}.
\end{equation}
Equivalently,
\begin{equation}\label{eq:visual-log-asymptotic}
\lim_{d\to\infty}
\frac{\log\Gamma_p^\K(r;m,n)}{\log d}
=
\abs{\frac1p-\frac12}.
\end{equation}
\end{corollary}

The next result describes the case in which the exponent depends on the dimensions.

\begin{maintheorem}\label{thm:D}
Let $(r_N)$, $(m_N)$ and $(n_N)$ be sequences in $\N$ such
that
\[
d_N:=\min\{r_N,m_N,n_N\}\longrightarrow\infty,
\]
and let $p_N\in[1,\infty]$. Put
\begin{equation}\label{eq:tauN}
\tau_N:=
\abs{\frac1{p_N}-\frac12}\log d_N.
\end{equation}
Then the following statements hold.

\begin{enumerate}[label=\textup{(\roman*)}]
\item The constants stay uniformly bounded precisely when
$(\tau_N)$ is bounded:
\[
\sup_N\G_{p_N}^{\K}(r_N;m_N,n_N)<\infty
\quad\Longleftrightarrow\quad
\sup_N\tau_N<\infty.
\]

\item If $\tau_N\to0$, then
\begin{equation}\label{eq:subcritical-KG}
\G_{p_N}^{\K}(r_N;m_N,n_N)
\longrightarrow K_G^{\K}.
\end{equation}

\item If $\tau_N\to\infty$, then
\begin{equation}\label{eq:supercritical-log}
\frac{\log\G_{p_N}^{\K}(r_N;m_N,n_N)}
{\tau_N}
\longrightarrow1.
\end{equation}

\item Suppose that $\tau_N\to\tau\in[0,\infty)$, and write
\[
\underline B_{\K}:=\liminf_{d\to\infty}B_d^{\K}.
\]
Then
\begin{equation}\label{eq:critical-envelope}
\max\{K_G^{\K},e^\tau\underline B_{\K}\}
\le
\liminf_N\G_{p_N}^{\K}(r_N;m_N,n_N)
\le
\limsup_N\G_{p_N}^{\K}(r_N;m_N,n_N)
\le K_G^{\K}e^\tau.
\end{equation}
Moreover, $\underline B_{\K}\ge1$.
\end{enumerate}

Thus uniform boundedness is equivalent to
\[
\abs{\frac1{p_N}-\frac12}
=O\!\left(\frac1{\log d_N}\right).
\]
More specifically, the transition window $\tau_N\asymp1$
corresponds to
\begin{equation}\label{eq:critical-scale}
\abs{\frac1{p_N}-\frac12}
\asymp\frac1{\log d_N}.
\end{equation}
When $p_N\to2$, the latter condition is equivalent to
$|p_N-2|\asymp1/\log d_N$.
\end{maintheorem}

We also give an interpretation of the endpoint in terms of tensor norm ratios and summing operators.

 For an operator $T:X\to Y$, we denote by
\[
 \pi_1^{(d)}(T):=\sup\left\{\sum_{j=1}^d\norm{Tx_j}:
 \sup_{x^*\in B_{X^*}}\sum_{j=1}^d\abs{x^*(x_j)}\le1\right\}
\]
the restricted $1$-summing norm with $d$ vectors, and $I_X:X\to X$ denotes the identity operator.

\begin{proposition}\label{prop:endpoint}
For both scalar fields,
\begin{equation}\label{eq:endpoint-rho}
 \G_\infty^{\K}(d;d,d)
=\rho(\ell_1^d(\K),\ell_1^d(\K)),
\end{equation}
For $d\ge1$ set
\begin{equation}\label{eq:Md-def}
 M_d:=2^{-d}\sum_{\varepsilon\in\{-1,1\}^d}
 \abs{\varepsilon_1+\cdots+\varepsilon_d}.
\end{equation}
Over $\R$,
\begin{equation}\label{eq:endpoint-real}
 B_d^{\R}
 =\frac1{\sqrt d}\pi_1^{(d)}(I_{\ell_1^d})
 \le
 \frac{\sqrt d}{M_d},
\end{equation}
and therefore
\begin{equation}\label{eq:endpoint-real-limsup}
 \limsup_{d\to\infty}B_d^{\R}
 \le\sqrt{\frac\pi2}.
\end{equation}
Over $\C$,
\begin{equation}\label{eq:endpoint-complex}
 1\le B_d^{\C}\le\frac2{\sqrt\pi}.
\end{equation}
\end{proposition}

In the complex case, complex interpolation gives an additional result for the square case. For $d\ge2$ put $L_d=\frac12\log d$ and, for $|s|\le L_d$, define
\begin{equation}\label{eq:critical-p-complex}
 \frac1{p_d(s)}=\frac12-\frac{s}{\log d},
 \qquad
 F_d(s):=
 \log\G_{p_d(s)}^{\C}(d;d,d).
\end{equation}

\begin{proposition}\label{prop:complex-profile}
For every $d\ge2$, the function $F_d$ is even, convex and $1$-Lipschitz on $[-L_d,L_d]$. For $0\le s\le L_d$ set
\[
 H_d(s):=F_d(s)-s.
\]
Then $H_d$ is convex and nonincreasing, with
\[
 H_d(0)=\log\G_2^{\C}(d;d,d),
 \qquad
 H_d(L_d)=\log B_d^{\C}.
\]
Therefore,
\begin{align}
 \max\{\log\G_2^{\C}(d;d,d),\,|s|+\log B_d^{\C}\}
 &\le F_d(s)\label{eq:profile-lower}\\
 &\le |s|
 +\left(1-\frac{2|s|}{\log d}\right)
   \log\G_2^{\C}(d;d,d)
 +\frac{2|s|}{\log d}\log B_d^{\C}.
 \label{eq:profile-upper}
\end{align}
Every sequence $d_k\to\infty$ has a subsequence along which $F_{d_k}$ converges locally uniformly on $\R$ and $B_{d_k}^{\C}\to B\in[1,2/\sqrt\pi]$. Every such limit satisfies
\begin{equation}\label{eq:profile-limit}
 \max\{\log K_G^{\C},|s|+\log B\}
 \le F(s)
 \le \log K_G^{\C}+|s|.
\end{equation}
\end{proposition}

\textcolor{black}{The effective-dimension estimates in Theorems~\ref{thm:A} and~\ref{thm:B}, as well as the lower endpoint envelopes in Theorem~\ref{thm:C}, are independent of interpolation and apply over both scalar fields. Complex interpolation enters in the upper estimates of Theorem~\ref{thm:C} and in Proposition~\ref{prop:complex-profile}.}

\section{Basic reductions}

We begin with two elementary properties that will be used repeatedly.

\begin{lemma}\label{lem:duality}
For every $r,m,n\in\N$ and $1\le p\le\infty$,
\begin{equation}\label{eq:duality}
 \G_p^{\K}(r;m,n)
 =\G_{p'}^{\K}(r;n,m).
\end{equation}
\end{lemma}

\begin{proof}
For $A\in\K^{m\times n}$, transposing $A$ and interchanging the two vector families in \eqref{eq:G-rmn} gives
\[
 G_p^{(r)}(A)=G_{p'}^{(r)}(A^T).
\]
Moreover $\norm{A}_{\infty,1}=\norm{A^T}_{\infty,1}$ by the bilinear form in \eqref{eq:matrix-norm}. Taking the supremum over $A\ne0$ proves \eqref{eq:duality}.
\end{proof}

\begin{lemma}\label{lem:padding}
If $r'\le r$, $m'\le m$ and $n'\le n$, then
\begin{equation}\label{eq:padding}
 \G_p^{\K}(r;m,n)
 \ge
 \G_p^{\K}(r';m',n').
\end{equation}
\end{lemma}

\begin{proof}
Embed an $m'\times n'$ matrix into the upper left corner of an $m\times n$ matrix and fill the remaining entries with zero. Embed the admissible vectors of $\ell_p^{r'}$ and $\ell_{p'}^{r'}$ into the first $r'$ coordinates of $\ell_p^r$ and $\ell_{p'}^r$. Both the numerator in \eqref{eq:G-rmn} and the denominator in \eqref{eq:matrix-norm} are unchanged.
\end{proof}

\section{Proof of Theorem A}

Fix an admissible configuration. The vectors $x_1,\ldots,x_m$
span a subspace of dimension at most $\min\{r,m\}$. On this subspace,
the vectors $y_j$ define bounded linear functionals. Lewis's estimate
\cite{Lewis1978} compares this subspace with a Euclidean space through
an isomorphism whose distortion is controlled by its dimension.
Duality gives the corresponding estimate after interchanging the two
vector families.

We record the estimate of Lewis in the following standard form:
if $E$ is an $s$-dimensional subspace of an $L_p$ space, then
\begin{equation}\label{eq:Lewis-distance}
d(E,\ell_2^s)\le s^{|1/p-1/2|}.
\end{equation}
\begin{lemma}\label{lem:Lewis-reduction}
Let $A\in\K^{m\times n}$ and $1\le p\le\infty$. If
$s=\min\{r,m\}$, then
\begin{equation}\label{eq:Lewis-x}
G_p^{(r)}(A)\le s^{a(p)}G_2^{(s)}(A).
\end{equation}
Likewise, if $t=\min\{r,n\}$, then
\begin{equation}\label{eq:Lewis-y}
G_p^{(r)}(A)\le t^{a(p)}G_2^{(t)}(A).
\end{equation}
\end{lemma}

\begin{proof}
Fix admissible vectors
$x_1,\ldots,x_m\in B_{\ell_p^r}$ and
$y_1,\ldots,y_n\in B_{\ell_{p'}^r}$, and put
\[
E=\operatorname{span}\{x_1,\ldots,x_m\},
\qquad q=\dim E.
\]
If $q=0$, the expression to be estimated is zero. Otherwise
$q\le s$, and \eqref{eq:Lewis-distance} gives an isomorphism
$U:E\to\ell_2^q$ which can be normalized so that
\begin{equation}\label{eq:U-normalization}
\norm U\le1,
\qquad
\norm{U^{-1}}\le q^{a(p)}.
\end{equation}
Set $u_i=Ux_i$. Then $\norm{u_i}_2\le1$.

For each $j$, restriction of the pairing with $y_j$ defines a
functional on $E$:
\[
\varphi_j(x)=\pair{x}{y_j}.
\]
Hölder's inequality gives
$\norm{\varphi_j}_{E^*}\le\norm{y_j}_{p'}\le1$.
Consequently,
\[
\norm{\varphi_j\circ U^{-1}}_{(\ell_2^q)^*}
\le\norm{U^{-1}}\le q^{a(p)}.
\]
Identify $(\ell_2^q)^*$ with $\ell_2^q$ using the bilinear
pairing in \eqref{eq:G-rmn}. This identification is isometric
over both scalar fields. We may therefore write
\[
\varphi_j(U^{-1}u)=\pair{u}{v_j}
\quad(u\in\ell_2^q),
\qquad
\norm{v_j}_2\le q^{a(p)}.
\]
The vectors $w_j=q^{-a(p)}v_j$ belong to $B_{\ell_2^q}$.
It follows that
\begin{align*}
\abs{\sum_{i=1}^m\sum_{j=1}^n
a_{ij}\pair{x_i}{y_j}}
&=
q^{a(p)}
\abs{\sum_{i=1}^m\sum_{j=1}^n
a_{ij}\pair{u_i}{w_j}}\\
&\le q^{a(p)}G_2^{(q)}(A)\\
&\le s^{a(p)}G_2^{(s)}(A).
\end{align*}
The last step uses $q\le s$ and zero padding. Taking the
supremum over the original vectors proves \eqref{eq:Lewis-x}.

To obtain \eqref{eq:Lewis-y}, apply \eqref{eq:Lewis-x} to
$A^T$ with exponent $p'$. Indeed,
$G_p^{(r)}(A)=G_{p'}^{(r)}(A^T)$ and $a(p')=a(p)$.
\end{proof}

\begin{proof}[Proof of Theorem \ref{thm:A}]
Put $d=\min\{r,m,n\}$. One of the two quantities
$\min\{r,m\}$ and $\min\{r,n\}$ equals $d$. Applying the
corresponding estimate from Lemma \ref{lem:Lewis-reduction}
gives, for every $A\in\K^{m\times n}$,
\begin{equation}\label{eq:A-pointwise-first}
G_p^{(r)}(A)\le d^{a(p)}G_2^{(d)}(A).
\end{equation}

At $p=2$, this says $G_2^{(r)}(A)\le G_2^{(d)}(A)$.
The reverse inequality follows by embedding $\ell_2^d$ into
$\ell_2^r$. Hence
\[
G_2^{(r)}(A)=G_2^{(d)}(A),
\]
which proves \eqref{eq:Hilbert-saturation} and completes
\eqref{eq:pointwise-bottleneck}. Dividing
\eqref{eq:A-pointwise-first} by $\norm A_{\infty,1}$ and taking
the supremum over nonzero $A$ yields
\begin{equation}\label{eq:A-upper-refined}
\G_p^{\K}(r;m,n)
\le d^{a(p)}\G_2^{\K}(r;m,n)
\le d^{a(p)}\kappa_G^{\K}(d).
\end{equation}

For the lower bound, zero padding gives
\begin{equation}\label{eq:square-padding}
\G_p^{\K}(r;m,n)\ge\G_p^{\K}(d;d,d).
\end{equation}
Suppose first that $p\ge2$. If
$x_i\in B_{\ell_\infty^d}$ and $y_j\in B_{\ell_1^d}$,
then
\[
\norm{d^{-1/p}x_i}_p\le1,
\qquad
\norm{y_j}_{p'}\le1.
\]
Thus every configuration admissible for the square
$p=\infty$ problem yields, after rescaling the $x_i$,
an admissible configuration for the $p$ problem. Taking
suprema gives
\begin{align}
\G_p^{\K}(d;d,d)
&\ge d^{-1/p}\G_\infty^{\K}(d;d,d)\notag\\
&=B_d^{\K}d^{1/2-1/p}
 =B_d^{\K}d^{a(p)}.
\label{eq:A-lower-pge2}
\end{align}
For $p\le2$, use Lemma \ref{lem:duality} and apply the same
argument to $p'\ge2$; the square dimensions remain unchanged.
Together with \eqref{eq:square-padding}, this proves
\eqref{eq:main-bottleneck}.
\end{proof}

\begin{corollary}\label{cor:d1}
If $\min\{r,m,n\}=1$, then
$\G_p^{\K}(r;m,n)=1$ for every $1\le p\le\infty$.
\end{corollary}

\begin{proof}
In this case $B_1^{\K}=\kappa_G^{\K}(1)=1$, so both sides of
\eqref{eq:main-bottleneck} equal one.
\end{proof}
\section{Endpoint compression and rectangular refinements}

The endpoint numerator admits an exact rectangular formula.
Write
\[
\mathbb T_{\K}:=\{\zeta\in\K:|\zeta|=1\}.
\]

\begin{lemma}\label{lem:rectangular-infinity}
Let $A=(a_{ij})\in\K^{m\times n}$ and $r\in\N$. Then
\begin{equation}\label{eq:rectangular-infinity-exact}
G_\infty^{(r)}(A)
=
\max_{\substack{
\sigma:\{1,\ldots,n\}\to\{1,\ldots,r\}\\
\zeta_1,\ldots,\zeta_n\in\mathbb T_{\K}}}
\sum_{i=1}^{m}\sum_{k=1}^{r}
\left|
\sum_{\substack{1\le j\le n\\ \sigma(j)=k}}
a_{ij}\zeta_j
\right|.
\end{equation}
In particular, if $r\ge n$, then
\begin{equation}\label{eq:rectangular-infinity-saturation}
G_\infty^{(r)}(A)
=
\sum_{i=1}^{m}\sum_{j=1}^{n}|a_{ij}|.
\end{equation}
Thus, for $A\in\K^{d\times d}$,
\begin{equation}\label{eq:G-infinity-exact}
G_\infty^{(d)}(A)
=
\sum_{i,j=1}^{d}|a_{ij}|.
\end{equation}
\end{lemma}

\begin{proof}
Fix $y_1,\ldots,y_n\in B_{\ell_1^r}$ and set
\[
b_{ik}:=\sum_{j=1}^{n}a_{ij}(y_j)_k
\qquad(1\le i\le m,\ 1\le k\le r).
\]
Since the coordinates of the vectors
$x_i\in B_{\ell_\infty^r}$ can be chosen independently,
\[
\sup_{x_1,\ldots,x_m\in B_{\ell_\infty^r}}
\left|\sum_{i=1}^{m}\sum_{k=1}^{r}x_{ik}b_{ik}\right|
=
\sum_{i=1}^{m}\sum_{k=1}^{r}|b_{ik}|.
\]
Indeed, the upper bound is the triangle inequality. For the
reverse inequality, take $x_{ik}=\overline{b_{ik}}/|b_{ik}|$
when $b_{ik}\ne0$ in the complex case, and the corresponding
sign in the real case; choose any scalar of modulus one when
$b_{ik}=0$. Consequently,
\[
G_\infty^{(r)}(A)
=
\max_{y_1,\ldots,y_n\in B_{\ell_1^r}}
\sum_{i=1}^{m}\sum_{k=1}^{r}
\left|\sum_{j=1}^{n}a_{ij}(y_j)_k\right|.
\]

The function on the right is continuous and convex in each
$y_j$ separately. Its maximum is attained because the domain
is compact. Starting from a maximizing tuple, maximize
successively over each $y_j$ while keeping the other variables
fixed. A convex function on the finite-dimensional $\ell_1^r$
unit ball has a maximizer at an extreme point. These extreme
points are precisely $\zeta e_k$, where
$\zeta\in\mathbb T_{\K}$ and $1\le k\le r$. Each successive
replacement preserves the global maximum. We may therefore
take $y_j=\zeta_j e_{\sigma(j)}$ for all $j$, which gives
\eqref{eq:rectangular-infinity-exact}.

The triangle inequality bounds the right-hand side of
\eqref{eq:rectangular-infinity-exact} by
$\sum_{i,j}|a_{ij}|$. If $r\ge n$, choose $\sigma$ injective;
then every inner sum has at most one term and this upper
bound is attained. This proves
\eqref{eq:rectangular-infinity-saturation}, of which
\eqref{eq:G-infinity-exact} is the square case.
\end{proof}

\begin{proposition}[Endpoint column compression]
\label{prop:endpoint-compression}
Let $r,m,n\in\N$ and put $k=\min\{r,n\}$. Then
\begin{equation}\label{eq:endpoint-compression}
\G_\infty^\K(r;m,n)
=
\rho\bigl(\ell_1^m(\K),\ell_1^k(\K)\bigr).
\end{equation}
Consequently,
\begin{equation}\label{eq:endpoint-compression-dual}
\G_1^\K(r;m,n)
=
\rho\bigl(\ell_1^n(\K),
          \ell_1^{\min\{r,m\}}(\K)\bigr).
\end{equation}
\end{proposition}

\begin{proof}
For a matrix $C=(c_{i\ell})\in\K^{m\times k}$, the canonical
identifications for $\ell_1$ tensor products give
\begin{equation}\label{eq:matrix-tensor-ratio}
\rho\bigl(\ell_1^m(\K),\ell_1^k(\K)\bigr)
=
\sup_{C\ne0}
\frac{\sum_{i=1}^m\sum_{\ell=1}^k|c_{i\ell}|}
     {\norm{C}_{\infty,1}}.
\end{equation}

Fix $A=(a_{ij})\in\K^{m\times n}$. By
Lemma~\ref{lem:rectangular-infinity}, there are an assignment
$\sigma:\{1,\ldots,n\}\to\{1,\ldots,r\}$ and phases
$\zeta_j\in\mathbb T_\K$ such that
\[
G_\infty^{(r)}(A)
=
\sum_{i=1}^m\sum_{\ell=1}^r|b_{i\ell}|,
\qquad
b_{i\ell}:=
\sum_{\sigma(j)=\ell}a_{ij}\zeta_j.
\]
At most $k$ columns of $B=(b_{i\ell})$ are nonzero. Deleting
the remaining columns does not change either
$\sum_{i,\ell}|b_{i\ell}|$ or $\norm{B}_{\infty,1}$.

Furthermore,
\begin{equation}\label{eq:compressed-matrix-norm}
\norm{B}_{\infty,1}\le\norm{A}_{\infty,1}.
\end{equation}
Indeed, for $\norm{t}_\infty\le1$, put
$u_j=\zeta_jt_{\sigma(j)}$. Then $\norm{u}_\infty\le1$ and
$Bt=Au$. Thus \eqref{eq:matrix-tensor-ratio} and
\eqref{eq:compressed-matrix-norm} give
\[
G_\infty^{(r)}(A)
\le
\rho\bigl(\ell_1^m(\K),\ell_1^k(\K)\bigr)
\norm{A}_{\infty,1}.
\]
Taking the supremum over $A\ne0$ proves one inequality in
\eqref{eq:endpoint-compression}.

Conversely, let $C\in\K^{m\times k}$ and append $n-k$ zero
columns to obtain a matrix $A\in\K^{m\times n}$. Then
$\norm{A}_{\infty,1}=\norm{C}_{\infty,1}$. Since $k\le r$,
the $k$ nonzero columns may be assigned to distinct coordinates
in \eqref{eq:rectangular-infinity-exact}. Hence
\[
G_\infty^{(r)}(A)
=
\sum_{i=1}^m\sum_{\ell=1}^k|c_{i\ell}|.
\]
Taking the supremum over $C\ne0$ proves the reverse inequality.

Finally, Lemma~\ref{lem:duality} gives
\[
\G_1^\K(r;m,n)=\G_\infty^\K(r;n,m).
\]
Applying \eqref{eq:endpoint-compression} to the right-hand
side proves \eqref{eq:endpoint-compression-dual}.
\end{proof}

\begin{lemma}[Operator realization]\label{lem:operator-realization}
For $A=(a_{ij})\in\K^{m\times n}$ and $1\le p\le\infty$, define
\[
T_A^{(p)}:
\ell_\infty^m(\ell_p^r)
\longrightarrow
\ell_1^n(\ell_p^r)
\]
by
\begin{equation}\label{eq:TAp-def}
T_A^{(p)}(x_1,\ldots,x_m)
:=
\left(
\sum_{i=1}^m a_{ij}x_i
\right)_{j=1}^n.
\end{equation}
Then
\begin{equation}\label{eq:TAp-Gp}
\|T_A^{(p)}\|
=
G_p^{(r)}(A).
\end{equation}
\end{lemma}

\begin{proof}
Fix $x_1,\ldots,x_m\in B_{\ell_p^r}$. Since the vectors
$y_1,\ldots,y_n\in B_{\ell_{p'}^r}$ may be chosen independently,
duality between $\ell_p^r$ and $\ell_{p'}^r$ gives
\begin{align*}
&\sup_{y_1,\ldots,y_n\in B_{\ell_{p'}^r}}
\left|
\sum_{i=1}^m\sum_{j=1}^n
a_{ij}\pair{x_i}{y_j}
\right|\\
&\hspace{2cm}
=
\sum_{j=1}^n
\left\|
\sum_{i=1}^m a_{ij}x_i
\right\|_p.
\end{align*}
Taking the supremum over
$x_1,\ldots,x_m\in B_{\ell_p^r}$ gives
\eqref{eq:TAp-Gp}.
\end{proof}

\begin{proof}[\textbf{Proof of Theorem~\ref{thm:C}}]
We begin with the lower estimates, which are valid over both scalar
fields, for the full range $1\le p\le\infty$, and require no
interpolation.

Fix $1\le p\le\infty$ and
\[
1\le k\le k_+=\min\{r,n\}.
\]
By Proposition~\ref{prop:endpoint-compression},
\begin{equation}\label{eq:C-endpoint-k-plus}
\Gamma_\infty^\K(k;m,k)
=
\mathcal R_{\K}(m,k).
\end{equation}
Let $C\in\K^{m\times k}$ be arbitrary. Every configuration
\[
x_1,\ldots,x_m\in B_{\ell_\infty^k},
\qquad
y_1,\ldots,y_k\in B_{\ell_1^k}
\]
which is admissible for $G_\infty^{(k)}(C)$ produces an admissible
configuration for the $p$-problem after replacing
\[
x_i
\quad\text{by}\quad
k^{-1/p}x_i.
\]
Indeed,
\[
\|k^{-1/p}x_i\|_p
\le
k^{-1/p}k^{1/p}\|x_i\|_\infty
\le1,
\]
whereas
\[
\|y_j\|_{p'}
\le
\|y_j\|_1
\le1.
\]
Therefore
\[
G_p^{(k)}(C)
\ge
k^{-1/p}G_\infty^{(k)}(C).
\]
Dividing by $\|C\|_{\infty,1}$ and taking the supremum over
$C\ne0$ gives
\begin{equation}\label{eq:C-k-lower-plus}
\Gamma_p^\K(k;m,k)
\ge
k^{-1/p}\mathcal R_{\K}(m,k).
\end{equation}
Since $k\le r$ and $k\le n$, zero padding yields
\[
\Gamma_p^\K(r;m,n)
\ge
\Gamma_p^\K(k;m,k).
\]
Taking the supremum over $1\le k\le k_+$ proves
\eqref{eq:rectangular-lower-plus}. The normalized form
\eqref{eq:rectangular-lower-plus-normalized} is merely the identity
\[
k^{-1/p}
=
k^{-1/2}k^{1/2-1/p}.
\]

For the second lower estimate, apply Lemma~\ref{lem:duality} for an
arbitrary $1\le p\le\infty$:
\[
\Gamma_p^\K(r;m,n)
=
\Gamma_{p'}^\K(r;n,m).
\]
The first lower estimate, already proved for the full range of
exponents, applied to the right-hand side gives
\[
\Gamma_p^\K(r;m,n)
\ge
\sup_{1\le k\le\min\{r,m\}}
k^{-1/p'}\mathcal R_{\K}(n,k),
\]
which is \eqref{eq:rectangular-lower-minus}. Since
\[
-\frac1{p'}
=
-\frac12+\left(\frac1p-\frac12\right),
\]
we also obtain
\eqref{eq:rectangular-lower-minus-normalized}.

We turn to the complex upper estimates. Let $2\le p\le\infty$ and put
\[
\theta:=1-\frac2p.
\]
Then $0\le\theta\le1$ and
\[
\frac1p
=
\frac{1-\theta}{2}
+
\frac{\theta}{\infty}.
\]
For a fixed complex matrix $A$, Lemma~\ref{lem:operator-realization}
and complex interpolation of the finite $\ell_p$ sums in
\eqref{eq:TAp-def} give
\begin{equation}\label{eq:C-fixedA-interpolation-plus}
G_p^{(r)}(A)
\le
G_2^{(r)}(A)^{1-\theta}
G_\infty^{(r)}(A)^\theta.
\end{equation}
Since $1-\theta=2/p$, division by $\|A\|_{\infty,1}$ gives
\begin{align*}
\frac{G_p^{(r)}(A)}{\|A\|_{\infty,1}}
&\le
\left(
\frac{G_2^{(r)}(A)}{\|A\|_{\infty,1}}
\right)^{2/p}
\left(
\frac{G_\infty^{(r)}(A)}{\|A\|_{\infty,1}}
\right)^{1-2/p}\\
&\le
\Gamma_2^\C(r;m,n)^{2/p}
\Gamma_\infty^\C(r;m,n)^{1-2/p}.
\end{align*}
Taking the supremum over $A\ne0$ and using
Proposition~\ref{prop:endpoint-compression},
\[
\Gamma_\infty^\C(r;m,n)
=
\mathcal R_{\C}(m,k_+),
\]
proves \eqref{eq:rectangular-upper-plus}.

Since for $p\ge2$
\[
2a(p)=1-\frac2p,
\]
formula \eqref{eq:rectangular-upper-plus-ratio} is simply a
rewriting of \eqref{eq:rectangular-upper-plus}.

Finally let $1\le p\le2$, and put
\[
\eta:=2-\frac2p.
\]
Then $0\le\eta\le1$ and
\[
\frac1p
=
\frac{1-\eta}{1}
+
\frac{\eta}{2}.
\]
Complex interpolation between $p=1$ and $p=2$ gives, for every
complex matrix $A$,
\[
G_p^{(r)}(A)
\le
G_1^{(r)}(A)^{1-\eta}
G_2^{(r)}(A)^\eta.
\]
Since
\[
1-\eta=\frac2p-1,
\qquad
\eta=2-\frac2p,
\]
we obtain
\[
\Gamma_p^\C(r;m,n)
\le
\Gamma_1^\C(r;m,n)^{\,2/p-1}
\Gamma_2^\C(r;m,n)^{\,2-2/p}.
\]
By duality and endpoint compression,
\[
\Gamma_1^\C(r;m,n)
=
\mathcal R_{\C}(n,k_-).
\]
This proves \eqref{eq:rectangular-upper-minus}. Since now
\[
2a(p)=\frac2p-1,
\]
the ratio form \eqref{eq:rectangular-upper-minus-ratio} follows as
well.
\end{proof}

\begin{remark}[Universal power versus rectangular profile]
\label{rem:power-versus-profile}
Theorem~\ref{thm:A} and Theorem~\ref{thm:C} contain different kinds
of information. The former identifies the universal dimensional power
through
\[
d=\min\{r,m,n\}.
\]
The latter retains the geometry of the whole rectangle through
\[
\mathcal R_{\K}(m,k),
\qquad
\mathcal R_{\K}(n,k),
\]
and optimization over intermediate endpoint dimensions $k$.

For example, when $p\ge2$, choosing $k=d$ in
\eqref{eq:rectangular-lower-plus} gives
\[
\Gamma_p^\K(r;m,n)
\ge
d^{-1/p}\mathcal R_{\K}(m,d)
\ge
d^{-1/p}\mathcal R_{\K}(d,d)
=
B_d^\K d^{1/2-1/p},
\]
which is the lower estimate of Theorem~\ref{thm:A}. The optimization
in Theorem~\ref{thm:C}, however, need not be attained at $k=d$.
Thus the effective dimension determines the power law without
exhausting the finer rectangular dependence.
\end{remark}

\begin{corollary}[Endpoint-sensitive improvement criterion]
\label{cor:endpoint-sensitive-improvement}
Assume $\K=\C$.

If $2\le p\le\infty$, then
\[
\Gamma_p^\C(r;m,n)
\le
d^{1/2-1/p}\Gamma_2^\C(r;m,n)\textcolor{black}{.}
\]
\textcolor{black}{If $2<p\le\infty$, the endpoint-interpolation estimate
\eqref{eq:rectangular-upper-plus} is strictly smaller whenever}
\[
\mathcal R_{\C}(m,\min\{r,n\})
<
\sqrt d\,\Gamma_2^\C(r;m,n).
\]

\textcolor{black}{If $1\le p<2$, the corresponding improvement occurs whenever}
\[
\mathcal R_{\C}(n,\min\{r,m\})
<
\sqrt d\,\Gamma_2^\C(r;m,n).
\]
\end{corollary}

\begin{proof}
For $p\ge2$, Theorem~\ref{thm:A} at $p=\infty$ gives
\[
\mathcal R_{\C}(m,\min\{r,n\})
=
\Gamma_\infty^\C(r;m,n)
\le
\sqrt d\,\Gamma_2^\C(r;m,n).
\]
Substitution into \eqref{eq:rectangular-upper-plus} yields
\[
\Gamma_p^\C(r;m,n)
\le
\Gamma_2^\C(r;m,n)^{2/p}
\left(
\sqrt d\,\Gamma_2^\C(r;m,n)
\right)^{1-2/p},
\]
which simplifies to
\[
\Gamma_p^\C(r;m,n)
\le
d^{1/2-1/p}\Gamma_2^\C(r;m,n).
\]
Strictness follows from strictness at the endpoint \textcolor{black}{whenever
$2<p\le\infty$}. The case $p<2$ follows by duality.
\end{proof}

\begin{proposition}[The complex two-row endpoint]
\label{prop:complex-two-rows}
For $r,n\in\N$ and $k=\min\{r,n\}$,
\begin{equation}\label{eq:complex-two-rows-infinity}
\G_\infty^\C(r;2,n)
=
k\sin\!\left(\frac{\pi}{2k}\right).
\end{equation}
If $r\ge2$, then
\begin{equation}\label{eq:complex-two-rows-one}
\G_1^\C(r;2,n)
=
n\sin\!\left(\frac{\pi}{2n}\right).
\end{equation}
\end{proposition}

\begin{proof}
By Proposition~\ref{prop:endpoint-compression}, the first
identity reduces to
\begin{equation}\label{eq:rho-complex-two-rows}
\rho\bigl(\ell_1^2(\C),\ell_1^k(\C)\bigr)
=
k\sin\!\left(\frac{\pi}{2k}\right).
\end{equation}
Let $A=(a_{ij})\in\C^{2\times k}$ and put
$c_j=|a_{1j}|+|a_{2j}|$. Choose $\phi_j\in\R$ so that
\[
\operatorname{Re}
 \bigl(e^{-it}a_{1j}\overline{a_{2j}}\bigr)
=
|a_{1j}a_{2j}|\cos(t-\phi_j)
\qquad(t\in\R);
\]
when $a_{1j}a_{2j}=0$, choose $\phi_j$ arbitrarily.
Writing $u=|a_{1j}|$ and $v=|a_{2j}|$, we have
\begin{align*}
|a_{1j}+e^{it}a_{2j}|^2
&=u^2+v^2+2uv\cos(t-\phi_j)\\
&=(u-v)^2+4uv\cos^2\!\left(
                 \frac{t-\phi_j}{2}\right)\\
&\ge
(u+v)^2\cos^2\!\left(
                 \frac{t-\phi_j}{2}\right).
\end{align*}
Consequently,
\begin{equation}\label{eq:two-row-column-bound}
|a_{1j}+e^{it}a_{2j}|
\ge
c_j\left|\cos\!\left(
               \frac{t-\phi_j}{2}\right)\right|.
\end{equation}

Put
\[
v_j=c_j
\left(\cos\frac{\phi_j}{2},
      \sin\frac{\phi_j}{2}\right)\in\R^2.
\]
The sharp weighted planar signed-sum inequality
\cite[Corollary~3]{AmbrusGrundbacher2026} gives
\[
\max_{\varepsilon_j\in\{-1,1\}}
\left\|\sum_{j=1}^k\varepsilon_jv_j\right\|_2
\ge
\frac{\sum_{j=1}^k c_j}
     {k\sin(\pi/(2k))}.
\]
The left-hand side equals
\[
\max_{\theta\in\R}
\sum_{j=1}^k c_j
\left|\cos\!\left(
       \theta-\frac{\phi_j}{2}\right)\right|.
\]
Indeed, Euclidean duality expresses the norm of each signed sum
as a maximum over unit vectors, and for each fixed unit vector
the signs can be chosen separately. Since
\[
\norm{A}_{\infty,1}
=
\max_{|z_1|=|z_2|=1}
\sum_{j=1}^k|a_{1j}z_1+a_{2j}z_2|,
\]
inequality \eqref{eq:two-row-column-bound} implies
\[
\norm{A}_{\infty,1}
\ge
\frac{\sum_{j=1}^k c_j}
     {k\sin(\pi/(2k))}.
\]
By \eqref{eq:matrix-tensor-ratio}, this proves the upper
bound in \eqref{eq:rho-complex-two-rows}.

For the reverse bound, let $\omega=e^{2\pi i/k}$ and
\[
A_k=
\begin{pmatrix}
1&1&\cdots&1\\
1&\omega&\cdots&\omega^{k-1}
\end{pmatrix}.
\]
Here $\sum_{i,j}|(A_k)_{ij}|=2k$. Moreover,
\begin{align*}
\norm{A_k}_{\infty,1}
&=
\max_{t\in\R}
\sum_{j=0}^{k-1}|1+e^{it}\omega^j|\\
&=
2\max_{t\in\R}
\sum_{j=0}^{k-1}
\left|\cos\!\left(
 \frac{t}{2}+\frac{\pi j}{k}\right)\right|\\
&=
\frac{2}{\sin(\pi/(2k))}.
\end{align*}
The last equality follows by summing the cosines between
successive zeros; equivalently, the unoriented directions
$\pi j/k$ form the equality configuration of the sharp
planar inequality. Thus
\[
\frac{\sum_{i,j}|(A_k)_{ij}|}
     {\norm{A_k}_{\infty,1}}
=
k\sin\!\left(\frac{\pi}{2k}\right),
\]
which proves \eqref{eq:rho-complex-two-rows}.

For $r\ge2$, Lemma~\ref{lem:duality} and
\eqref{eq:endpoint-compression} give
\[
\G_1^\C(r;2,n)
=
\rho\bigl(\ell_1^n(\C),\ell_1^2(\C)\bigr).
\]
The symmetry of $\rho$ and \eqref{eq:rho-complex-two-rows},
with $k=n$, prove \eqref{eq:complex-two-rows-one}.
\end{proof}

In particular,
\begin{equation}\label{eq:complex-223-endpoint}
\G_\infty^\C(2;2,3)=\sqrt2,
\qquad
\G_\infty^\C(r;2,3)=\frac32
\quad(r\ge3).
\end{equation}
The planar Euclidean value used above was obtained earlier in
\cite{AmbrusNietert2020,FuWangYan2023}; the weighted formulation
is stated in \cite{AmbrusGrundbacher2026}.

\begin{proposition}[Two-dimensional exact constants]
\label{prop:small-exact}
Let $1\le p\le\infty$.
\begin{enumerate}[label=\textup{(\roman*)}]
\item If $\min\{r,m,n\}=2$, then
\begin{equation}\label{eq:small-real-exact}
\G_p^\R(r;m,n)=\sqrt2\,2^{a(p)}.
\end{equation}
\item If $r\ge2$, then
\begin{equation}\label{eq:small-complex-exact}
\G_p^\C(r;2,2)=2^{a(p)}.
\end{equation}
\end{enumerate}
\end{proposition}

\begin{proof}
Over $\R$, the order-two Grothendieck constant in
\eqref{eq:kappaGd} is
$\kappa_G^\R(2)=\sqrt2$ \cite{Krivine1979,BagchiMisra2008}.
For the real $2\times2$ Hadamard matrix
\[
H_2=
\begin{pmatrix}
1&1\\
1&-1
\end{pmatrix},
\]
one has
\[
\norm{H_2}_{\infty,1}=2,
\qquad
G_\infty^{(2)}(H_2)=4.
\]
Thus $B_2^\R=\sqrt2$. If
$d=\min\{r,m,n\}=2$, both sides of
\eqref{eq:main-bottleneck} coincide and yield
\eqref{eq:small-real-exact}.

Over $\C$, the Hilbertian Grothendieck ratio for complex
$2\times2$ matrices is $1$ \cite{BagchiMisra2008}.
Theorem~\ref{thm:A} therefore gives
\[
\G_p^\C(r;2,2)\le 2^{a(p)}.
\]
On the other hand,
\[
\norm{H_2}_{\infty,1}=2\sqrt2
\]
when $H_2$ is regarded as a complex matrix. If $p\ge2$,
choose $y_1=e_1$, $y_2=e_2$ and
\[
x_1=2^{-1/p}(1,1),
\qquad
x_2=2^{-1/p}(1,-1).
\]
These vectors are admissible in two coordinates and give
\[
G_p^{(r)}(H_2)\ge 4\,2^{-1/p}.
\]
Consequently,
\[
\G_p^\C(r;2,2)
\ge
\frac{4\,2^{-1/p}}{2\sqrt2}
=
2^{1/2-1/p}
=
2^{a(p)}.
\]
For $p\le2$, apply Lemma~\ref{lem:duality} to the square
matrix problem. This proves \eqref{eq:small-complex-exact}.
\end{proof}

\begin{proposition}[Complex two-row interior bounds]
\label{prop:two-row-interior}
Let $r\ge2$, $n\ge1$, and put
\[
K:=\min\{r,n\}.
\]
For every $1\le p\le\infty$,
\begin{equation}\label{eq:two-row-interior-lower}
\Gamma_p^\C(r;2,n)
\ge
\max\left\{
 n\sin\!\left(\frac{\pi}{2n}\right)2^{1/p-1},
 \max_{1\le k\le K}
 k^{\,1-1/p}\sin\!\left(\frac{\pi}{2k}\right)
\right\}.
\end{equation}

If $1\le p\le2$, then
\begin{equation}\label{eq:two-row-interior-minus}
\Gamma_p^\C(r;2,n)
\le
\left[
 n\sin\!\left(\frac{\pi}{2n}\right)
\right]^{2/p-1}
\Gamma_2^\C(r;2,n)^{\,2-2/p}.
\end{equation}
If $2\le p\le\infty$, then
\begin{equation}\label{eq:two-row-interior-plus}
\Gamma_p^\C(r;2,n)
\le
\Gamma_2^\C(r;2,n)^{2/p}
\left[
 K\sin\!\left(\frac{\pi}{2K}\right)
\right]^{1-2/p}.
\end{equation}

In particular,
\begin{equation}\label{eq:two-row-Hilbert-lower}
\Gamma_2^\C(r;2,n)
\ge
\frac{n}{\sqrt2}
\sin\!\left(\frac{\pi}{2n}\right).
\end{equation}
Moreover, if $n\ge3$ and $1\le p\le2$, then
\begin{equation}\label{eq:two-row-interior-separation}
\Gamma_p^\C(r;2,n)
>
\Gamma_p^\C(r;2,2).
\end{equation}
\end{proposition}

\begin{proof}
Since $r\ge2$, one has
\[
k_-:=\min\{r,2\}=2.
\]
The second lower envelope in Theorem~\ref{thm:C}, now valid for every
$1\le p\le\infty$, gives
\[
\Gamma_p^\C(r;2,n)
\ge
2^{-1/p'}\mathcal R_{\C}(n,2).
\]
By symmetry of $\mathcal R_{\C}$ and
Proposition~\ref{prop:complex-two-rows},
\[
\mathcal R_{\C}(n,2)
=
\mathcal R_{\C}(2,n)
=
n\sin\!\left(\frac{\pi}{2n}\right).
\]
Since
\[
-\frac1{p'}=\frac1p-1,
\]
we obtain the first branch in
\eqref{eq:two-row-interior-lower}.

Likewise, the first lower envelope in Theorem~\ref{thm:C} gives, for
every $1\le k\le K$ and every $1\le p\le\infty$,
\[
\Gamma_p^\C(r;2,n)
\ge
k^{-1/p}\mathcal R_{\C}(2,k).
\]
Using again Proposition~\ref{prop:complex-two-rows},
\[
\mathcal R_{\C}(2,k)
=
k\sin\!\left(\frac{\pi}{2k}\right),
\]
so
\[
\Gamma_p^\C(r;2,n)
\ge
k^{1-1/p}\sin\!\left(\frac{\pi}{2k}\right).
\]
Taking the maximum over $1\le k\le K$ proves the second branch of
\eqref{eq:two-row-interior-lower}.

For $1\le p\le2$, the complex upper envelope in
Theorem~\ref{thm:C} gives
\[
\Gamma_p^\C(r;2,n)
\le
\mathcal R_{\C}(n,2)^{\,2/p-1}
\Gamma_2^\C(r;2,n)^{\,2-2/p},
\]
which is \eqref{eq:two-row-interior-minus}. For $2\le p\le\infty$,
the other complex upper envelope gives
\begin{align*}
\Gamma_p^\C(r;2,n)
&\le
\Gamma_2^\C(r;2,n)^{2/p}
\mathcal R_{\C}(2,K)^{1-2/p}\\
&=
\Gamma_2^\C(r;2,n)^{2/p}
\left[
K\sin\!\left(\frac{\pi}{2K}\right)
\right]^{1-2/p},
\end{align*}
which is \eqref{eq:two-row-interior-plus}.

Setting $p=2$ in the first branch of
\eqref{eq:two-row-interior-lower} gives
\eqref{eq:two-row-Hilbert-lower}.

Finally, for $n\ge3$ the function
\[
q\longmapsto
q\sin\!\left(\frac{\pi}{2q}\right)
\]
is strictly increasing for $q\ge1$. Hence, when $1\le p\le2$,
\begin{align*}
\Gamma_p^\C(r;2,n)
&\ge
n\sin\!\left(\frac{\pi}{2n}\right)2^{1/p-1}\\
&>
\sqrt2\,2^{1/p-1}\\
&=
2^{1/p-1/2}\\
&=
\Gamma_p^\C(r;2,2),
\end{align*}
where the last equality follows from
Proposition~\ref{prop:small-exact}.
\end{proof}

\begin{corollary}[Fully explicit complex two-row upper bounds]
\label{cor:two-row-explicit-upper}
Let $r\ge2$, $n\ge1$, and $K=\min\{r,n\}$. Then, for
$1\le p\le2$,
\begin{equation}\label{eq:two-row-explicit-upper-minus}
\Gamma_p^\C(r;2,n)
\le
\left[
 n\sin\!\left(\frac{\pi}{2n}\right)
\right]^{2/p-1}
\bigl[\kappa_G^\C(2)\bigr]^{2-2/p},
\end{equation}
and, for $2\le p\le\infty$,
\begin{equation}\label{eq:two-row-explicit-upper-plus}
\Gamma_p^\C(r;2,n)
\le
\bigl[\kappa_G^\C(2)\bigr]^{2/p}
\left[
 K\sin\!\left(\frac{\pi}{2K}\right)
\right]^{1-2/p}.
\end{equation}
\end{corollary}

\begin{proof}
By Theorem~\ref{thm:A}, the Hilbertian stabilization for two-row
matrices gives
\[
\Gamma_2^\C(r;2,n)
\le
\kappa_G^\C(2).
\]
Substitution into
\eqref{eq:two-row-interior-minus} and
\eqref{eq:two-row-interior-plus} proves the two estimates.
\end{proof}

\begin{remark}\label{rem:two-row-intermediate-k}
The first branch inside the maximum in
\eqref{eq:two-row-interior-lower} and the optimized $k$-branch arise
from the two different rectangular endpoint mechanisms in
Theorem~\ref{thm:C}. In particular, for $2<p<\infty$ the best endpoint
dimension in the $k$-branch may be an intermediate value
\[
1<k<K.
\]
The continuous analogue is
\[
f_p(x):=
x^{1-1/p}\sin\!\left(\frac{\pi}{2x}\right),
\qquad x\ge1.
\]
Writing
\[
z=\frac{\pi}{2x},
\]
an interior critical point satisfies
\begin{equation}\label{eq:two-row-continuous-critical}
z\cot z=1-\frac1p.
\end{equation}
For $p>1$ this equation has a unique solution
$z_p\in(0,\pi/2)$. Moreover, as $p\to\infty$,
\[
z\cot z
=1-\frac{z^2}{3}+O(z^4),
\]
so
\[
z_p\sim\sqrt{\frac3p}
\qquad\text{and hence}\qquad
x_p:=\frac{\pi}{2z_p}
\sim
\frac{\pi}{2\sqrt3}\sqrt p.
\]
Thus, whenever $K$ is sufficiently large relative to $\sqrt p$, the
strongest lower bound coming from the $k$-branch is generated by an
intermediate endpoint dimension on the natural scale $k\asymp\sqrt p$.
\end{remark}

\begin{remark}[Hilbertian and endpoint saturation differ]
\label{rem:rectangular-endpoint-saturation}
Let $\omega=e^{2\pi i/3}$ and consider
\[
A=
\begin{pmatrix}
1&1&1\\
1&\omega&\omega^2
\end{pmatrix}
\in\C^{2\times3}.
\]
Then
\begin{equation}\label{eq:rectangular-endpoint-example}
G_\infty^{(2)}(A)=2+2\sqrt3
<6=G_\infty^{(3)}(A).
\end{equation}
Indeed, the two rows of $A$ are orthogonal and have squared
Euclidean norm $3$, so $AA^*=3I_2$. By
\eqref{eq:rectangular-infinity-exact}, an assignment of the three
columns to two coordinates either places all three in one group
or places two in one group and the remaining column in the other.

In the first case, for $z\in\mathbb T_{\C}^3$,
\[
\|Az\|_1
\le\sqrt2\,\|Az\|_2
\le\sqrt2\,\|A\|_{2\to2}\|z\|_2
=3\sqrt2
<2+2\sqrt3.
\]
In the second case, the singleton column contributes $2$.
If $a_j,a_k$ are the two columns in the other group, then
$\|a_j\|_2^2=\|a_k\|_2^2=2$ and $|a_j^*a_k|=1$, where $a_j^*a_k$
denotes the usual Hermitian inner product on $\C^2$. Hence, for
$\zeta_j,\zeta_k\in\mathbb T_{\C}$,
\[
\|\zeta_j a_j+\zeta_k a_k\|_1
\le\sqrt2\,\|\zeta_j a_j+\zeta_k a_k\|_2
\le\sqrt2\,\sqrt{4+2}=2\sqrt3.
\]
This bound is attained by grouping the first two columns and
choosing $\zeta_1=1$, $\zeta_2=e^{-i\pi/3}$: both coordinates
of $\zeta_1a_1+\zeta_2a_2$ have modulus $\sqrt3$.
This proves the first equality in
\eqref{eq:rectangular-endpoint-example}. The second follows
from \eqref{eq:rectangular-infinity-saturation}, since $r=3=n$.

On the other hand, Theorem~\ref{thm:A} gives
$G_2^{(2)}(A)=G_2^{(3)}(A)$, because
$\min\{2,m,n\}=2$. Thus the Hilbertian stabilization at the
effective dimension need not be a pointwise stabilization of
$G_\infty^{(r)}(A)$.
\end{remark}

\begin{lemma}\label{lem:Hadamard-truncation}
Let $h\ge d$ and suppose that there exists a real Hadamard matrix $H$ of order $h$. If $A$ is any $d\times d$ corner of $H$, then
\begin{equation}\label{eq:corner-infty1}
 \norm{A}_{\infty,1}\le d\sqrt h.
\end{equation}
Therefore,
\begin{equation}\label{eq:corner-B}
 B_d^{\R}\ge\sqrt{\frac dh}.
\end{equation}
\end{lemma}

\begin{proof}
Let $z\in\R^d$ satisfy $\norm{z}_\infty\le1$ and extend it by zero to $\widetilde z\in\R^h$. If $P:\R^h\to\R^d$ denotes the coordinate projection corresponding to the selected rows, then
\[
 Az=PH\widetilde z.
\]
Since $HH^T=hI_h$,
\begin{align*}
 \norm{Az}_1
 &\le \sqrt d\,\norm{Az}_2
 \le \sqrt d\,\norm{H\widetilde z}_2\\
 &=\sqrt{dh}\,\norm{\widetilde z}_2
 \le d\sqrt h.
\end{align*}
Taking the supremum over $z$ proves \eqref{eq:corner-infty1}. Every entry of $A$ has modulus one, so \eqref{eq:G-infinity-exact} gives $G_\infty^{(d)}(A)=d^2$. Therefore
\[
 \G_\infty^{\R}(d;d,d)
 \ge\frac{d^2}{d\sqrt h}
 =\frac d{\sqrt h},
\]
and division by $\sqrt d$ proves \eqref{eq:corner-B}.
\end{proof}

\begin{proof}[\textbf{Proof of Theorem \ref{thm:B}}]
For the real field, apply Lemma \ref{lem:Hadamard-truncation} with
$h=h_+(d)$ to obtain
\begin{equation}\label{eq:real-hplus-lower}
B_d^{\R}\ge\sqrt{\frac{d}{h_+(d)}}=\beta_d^{\R}.
\end{equation}
To make the quantitative estimate explicit, write the Hadamard orders
in increasing order as $h_1<h_2<\cdots$ and define the Hadamard gap
function $\gamma(x):=\max\{h_{i+1}-h_i:h_i\le x\}$ for $x\ge0$.
Livinskyi's asymptotic existence result, in the form recorded by
Brent, Osborn and Smith \cite{BrentOsbornSmith2016}, implies that the gaps
between consecutive Hadamard orders satisfy
\begin{equation}\label{eq:Hadamard-gap-function}
\gamma(x)=O(x^{1/6})\qquad\text{as }x\to\infty.
\end{equation}
If $d$ is not itself a Hadamard order, choose consecutive Hadamard
orders $h_i<d<h_{i+1}$. Then $h_{i+1}=h_+(d)$ and
\[
0\le h_+(d)-d\le h_{i+1}-h_i\le\gamma(d)=O(d^{1/6}).
\]
If $d$ is already a Hadamard order, then $h_+(d)=d$ and $h_+(d)-d=0$.
Thus, unconditionally,
\begin{equation}\label{eq:Hadamard-gap-rate}
h_+(d)-d=O(d^{1/6}).
\end{equation}
Setting $x_d:=(h_+(d)-d)/d$, we have $x_d=O(d^{-5/6})$. Since $x_d\ge0$,
\begin{align*}
1-\beta_d^{\R}
&=1-(1+x_d)^{-1/2}\\
&=\frac{(1+x_d)-1}{\sqrt{1+x_d}\,(1+\sqrt{1+x_d})}\\
&=\frac{x_d}{\sqrt{1+x_d}\,(1+\sqrt{1+x_d})}\\
&\le\frac{x_d}{2}
=O(d^{-5/6}).
\end{align*}
If $d$ is a Hadamard order, $h_+(d)=d$ and \eqref{eq:real-hplus-lower}
gives $B_d^{\R}\ge1$ identically.

The gap estimate \eqref{eq:Hadamard-gap-rate} is needed only to obtain the quantitative factor $1-O(d^{-5/6})$.
The qualitative equivalence \eqref{eq:rectangular-order} requires only
elementary means: letting $2^k$ be the smallest power of two with
$2^k\ge d$, Sylvester's classical construction provides an explicit
Hadamard matrix of order $2^k\le 2d$. Hence $h_+(d)\le 2d$, which
immediately yields the universal lower bound $\beta_d^{\R}\ge 1/\sqrt2$
without invoking any results from analytic number theory.

For the complex field, let $\mathcal F_d=(\omega^{(j-1)(k-1)})_{j,k=1}^d$,
where $\omega=e^{2\pi i/d}$, be the Fourier matrix of order $d$.
Every entry of $\mathcal F_d$ has modulus one, and
$\mathcal F_d\mathcal F_d^*=dI_d$. For any $z\in\C^d$ with
$\norm z_\infty\le1$,
\begin{align*}
\norm{\mathcal F_d z}_1
&\le \sqrt d\,\norm{\mathcal F_d z}_2
 =\sqrt d\,\norm{\mathcal F_d}_{2\to2}\norm z_2\\
&=d\,\norm z_2
 \le d^{3/2}\norm z_\infty
 \le d^{3/2}.
\end{align*}
Taking the supremum over $\norm z_\infty\le1$ gives
$\norm{\mathcal F_d}_{\infty,1}\le d^{3/2}$. Together with
\eqref{eq:G-infinity-exact}, we obtain
\[
\G_\infty^{\C}(d;d,d)\ge\frac{d^2}{d^{3/2}}=\sqrt d,
\qquad\text{so}\quad B_d^{\C}\ge1.
\]
This completes the proof of \eqref{eq:Bd-lower} and \eqref{eq:beta-rate}.
Substituting \eqref{eq:Bd-lower} into Theorem \ref{thm:A} yields
\eqref{eq:uniform-rectangular}. Since $\kappa_G^{\K}(d)\le K_G^{\K}$, the
constants in \eqref{eq:rectangular-order} remain bounded independently
of $r,m,n$ and $p$.
\end{proof}

\begin{proof}[Proof of Corollary~\ref{cor:fixed-p-asymptotic}]
Let
\[
a(p)=\abs{\frac1p-\frac12}>0.
\]
By Theorem~\ref{thm:B},
\[
\beta_d^{\K}d^{a(p)}
\le
\Gamma_p^{\K}(r;m,n)
\le
\kappa_G^{\K}(d)d^{a(p)}.
\]
Taking logarithms gives
\[
a(p)\log d+\log\beta_d^{\K}
\le
\log\Gamma_p^{\K}(r;m,n)
\le
a(p)\log d+\log \kappa_G^{\K}(d).
\]
Since
\[
\beta_d^{\K}\longrightarrow1
\]
and
\[
1\le \kappa_G^{\K}(d)\le K_G^{\K},
\]
we have
\[
\frac{\log\beta_d^{\K}}{\log d}\longrightarrow0,
\qquad
\frac{\log \kappa_G^{\K}(d)}{\log d}\longrightarrow0.
\]
Hence
\[
\frac{\log\Gamma_p^{\K}(r;m,n)}{\log d}
\longrightarrow
a(p),
\]
which proves \eqref{eq:visual-log-asymptotic}.

Equivalently,
\[
\log\Gamma_p^{\K}(r;m,n)
=
\bigl(a(p)+o(1)\bigr)\log d,
\]
and therefore
\[
\Gamma_p^{\K}(r;m,n)
=
d^{a(p)+o(1)}.
\]
For fixed $p$, the error depends only on $d$ and is therefore uniform
over all triples $(r,m,n)$ satisfying
$\min\{r,m,n\}=d$.
\end{proof}

\begin{proof}[Proof of Proposition \ref{prop:endpoint}]
Identity \eqref{eq:endpoint-rho} is the square case of
Proposition~\ref{prop:endpoint-compression}.

Over $\R$, Proposition~12 of \cite{AubrunLamiPalazuelosSzarekWinter2020} gives
\[
 \rho(\ell_1^d,X)=\pi_1^{(d)}(I_X).
\]
For $X=\ell_1^d$, their formula is
\begin{equation}\label{eq:pi1-Md}
 \pi_1(I_{\ell_1^d})=\frac d{M_d},
\end{equation}
where $M_d$ is the finite average in \eqref{eq:Md-def}. Since
$\pi_1^{(d)}(I_{\ell_1^d})\le\pi_1(I_{\ell_1^d})$, we obtain
\begin{align*}
 B_d^{\R}
 &=\frac1{\sqrt d}\pi_1^{(d)}(I_{\ell_1^d})\\
 &\le\frac1{\sqrt d}\pi_1(I_{\ell_1^d})\\
 &=\frac{\sqrt d}{M_d}.
\end{align*}
The average $M_d$ has an elementary closed form. If $d=2m$, grouping sign vectors according to the number $k$ of negative coordinates gives
\begin{align*}
 M_{2m}
 &=2^{-2m}\sum_{k=0}^{2m}|2m-2k|\binom{2m}{k}\\
 &=2^{1-2m}\sum_{k=0}^{m-1}(2m-2k)\binom{2m}{k}\\
 &=2^{2-2m}\sum_{k=0}^{m-1}(m-k)\binom{2m}{k}.
\end{align*}
For $0\le k\le m-1$,
\[
 (m-k)\binom{2m}{k}
 =m\left[\binom{2m-1}{k}-\binom{2m-1}{k-1}\right],
\]
where $\binom{2m-1}{-1}=0$. The sum telescopes, and therefore
\begin{align}
 M_{2m}
 &=2^{2-2m}m\binom{2m-1}{m-1}\notag\\
 &=\frac{2m}{2^{2m}}\binom{2m}{m}.
 \label{eq:Md-even}
\end{align}
If $d=2m+1$, symmetry around $k=m+\frac12$ gives
\begin{align*}
 M_{2m+1}
 &=2^{-2m}\sum_{k=0}^{m}(2m+1-2k)\binom{2m+1}{k}.
\end{align*}
For $0\le k\le m$,
\[
 (2m+1-2k)\binom{2m+1}{k}
 =(2m+1)\left[\binom{2m}{k}-\binom{2m}{k-1}\right].
\]
A second telescoping sum gives
\begin{equation}\label{eq:Md-odd}
 M_{2m+1}=\frac{2m+1}{2^{2m}}\binom{2m}{m}.
\end{equation}
By Stirling's formula,
\[
 \binom{2m}{m}
 =\frac{4^m}{\sqrt{\pi m}}\bigl(1+O(m^{-1})\bigr).
\]
Substitution into \eqref{eq:Md-even} and \eqref{eq:Md-odd} yields, in both parity classes,
\begin{equation}\label{eq:Md-asymptotic}
 M_d=\sqrt{\frac{2d}{\pi}}\bigl(1+O(d^{-1})\bigr).
\end{equation}
Therefore
\[
 \limsup_{d\to\infty}B_d^{\R}
 \le
 \lim_{d\to\infty}\frac{\sqrt d}{M_d}
 =\sqrt{\frac\pi2},
\]
which proves \eqref{eq:endpoint-real-limsup}.

Over $\C$, let {$\mathbb T:=\{z\in\C:|z|=1\}$}, and let {$\sigma_d$} denote normalized Haar measure on $\mathbb T^d$. For every $z\in\mathbb T^d$ one has $\norm z_\infty=1$, and hence
\[
 \norm{A}_{\infty,1}\ge\norm{Az}_1.
\]
Integrating this pointwise inequality over $\mathbb T^d$ gives
\begin{align*}
 \norm{A}_{\infty,1}
 &\ge \int_{\mathbb T^d}\norm{Az}_1\,{d\sigma_d(z)}\\
 &=\sum_{i=1}^d
   \int_{\mathbb T^d}\abs{\sum_{j=1}^d a_{ij}z_j}\,{d\sigma_d(z)}.
\end{align*}
Sawa's sharp $L_1$ Steinhaus inequality \cite{Sawa1985} gives, row by row,
\[
 \int_{\mathbb T^d}\abs{\sum_{j=1}^d a_{ij}z_j}\,{d\sigma_d(z)}
 \ge\frac{\sqrt\pi}{2}
 \left(\sum_{j=1}^d|a_{ij}|^2\right)^{1/2}.
\]
Therefore
\begin{align*}
 \norm{A}_{\infty,1}
 &\ge\frac{\sqrt\pi}{2}
 \sum_{i=1}^d\left(\sum_{j=1}^d|a_{ij}|^2\right)^{1/2}\\
 &\ge\frac{\sqrt\pi}{2\sqrt d}
 \sum_{i=1}^d\sum_{j=1}^d|a_{ij}|,
\end{align*}
where the last inequality follows from
\[
 \sum_{j=1}^d|a_{ij}|
 \le\sqrt d\left(\sum_{j=1}^d|a_{ij}|^2\right)^{1/2}
\]
for each row $i$. Hence
\[
 \frac{G_\infty^{(d)}(A)}{\norm A_{\infty,1}}
 \le\frac{2\sqrt d}{\sqrt\pi}.
\]
Taking the supremum over $A\ne0$ and dividing by $\sqrt d$ gives
$B_d^{\C}\le2/\sqrt\pi$. The lower bound $B_d^{\C}\ge1$ was proved in Theorem \ref{thm:B}.
\end{proof}

\section{Proof of Theorem D}

We first prove the Hilbertian lower bound needed in the critical case.

\begin{lemma}\label{lem:fixed-Hilbert-extremizers}
Suppose $r_N,m_N,n_N\to\infty$ and $p_N\to2$. Then
\begin{equation}\label{eq:liminf-KG}
 \liminf_{N\to\infty}
 \G_{p_N}^{\K}(r_N;m_N,n_N)
 \ge K_G^{\K}.
\end{equation}
\end{lemma}

\begin{proof}
Fix $\eps>0$. By the definition of $K_G^{\K}$, there are integers $r_0,m_0,n_0$, a matrix $A\in\K^{m_0\times n_0}$, and vectors
\[
 u_1,\ldots,u_{m_0},
 v_1,\ldots,v_{n_0}\in B_{\ell_2^{r_0}}
\]
such that
\begin{equation}\label{eq:fixed-near-KG}
 \abs{\sum_{i,j}a_{ij}\pair{u_i}{v_j}}
 >(K_G^{\K}-\eps)\norm{A}_{\infty,1}.
\end{equation}
For large $N$, all three dimensions $r_N,m_N,n_N$ dominate $r_0,m_0,n_0$. Put
\[
 \alpha_N:=\max\left\{\frac1{p_N}-\frac12,0\right\},
 \qquad
 \alpha_N':=\max\left\{\frac1{p_N'}-\frac12,0\right\}.
\]
For every $u\in\K^{r_0}$, the finite-dimensional norm comparison gives
\[
 \norm u_{p_N}\le r_0^{\alpha_N}\norm u_2,
\]
and similarly
\[
 \norm v_{p_N'}\le r_0^{\alpha_N'}\norm v_2.
\]
Since $p_N\to2$, both $\alpha_N$ and $\alpha_N'$ tend to zero. Set
\[
 c_N:=r_0^{\alpha_N},
 \qquad
 c_N':=r_0^{\alpha_N'}.
\]
Then $c_N,c_N'\to1$, and the vectors $u_i/c_N$ and $v_j/c_N'$ are admissible in the $p_N$-problem after zero padding. Therefore \eqref{eq:fixed-near-KG} gives
\[
 \G_{p_N}^{\K}(r_N;m_N,n_N)
 \ge\frac{K_G^{\K}-\eps}{c_Nc_N'}.
\]
Letting $N\to\infty$ and then $\eps\downarrow0$ proves \eqref{eq:liminf-KG}.
\end{proof}

\begin{proof}[Proof of Theorem \ref{thm:D}]
Theorem \ref{thm:B} gives
\begin{equation}\label{eq:tau-master-rect}
 \beta_{d_N}^{\K}e^{\tau_N}
 \le
 \G_{p_N}^{\K}(r_N;m_N,n_N)
 \le
 \kappa_G^{\K}(d_N)e^{\tau_N}
 \le K_G^{\K}e^{\tau_N}.
\end{equation}
Suppose first that $\sup_N\tau_N<\infty$. The upper bound in \eqref{eq:tau-master-rect} gives
\[
 \sup_N\G_{p_N}^{\K}(r_N;m_N,n_N)
 \le K_G^{\K}\exp(\sup_N\tau_N)<\infty.
\]
Conversely, if $(\tau_N)$ is unbounded, choose a subsequence $(N_k)$ with $\tau_{N_k}\to\infty$. Since $\beta_{d_{N_k}}^{\K}\to1$, the lower bound in \eqref{eq:tau-master-rect} gives
\[
 \G_{p_{N_k}}^{\K}(r_{N_k};m_{N_k},n_{N_k})
 \ge \beta_{d_{N_k}}^{\K}e^{\tau_{N_k}}
 \longrightarrow\infty.
\]
This proves (i).

If $\tau_N\to0$, then $a(p_N)\to0$, hence $p_N\to2$. Since $d_N\to\infty$, all three dimensions tend to infinity. Lemma \ref{lem:fixed-Hilbert-extremizers} gives
\[
 \liminf_N\G_{p_N}^{\K}(r_N;m_N,n_N)
 \ge K_G^{\K}.
\]
The upper estimate in \eqref{eq:tau-master-rect}, together with \eqref{eq:kappaGd-limit}, gives
\[
 \limsup_N\G_{p_N}^{\K}(r_N;m_N,n_N)
 \le K_G^{\K}.
\]
This proves (ii).

If $\tau_N\to\infty$, take logarithms in \eqref{eq:tau-master-rect}:
\[
 1+\frac{\log\beta_{d_N}^{\K}}{\tau_N}
 \le
 \frac{\log\G_{p_N}^{\K}(r_N;m_N,n_N)}{\tau_N}
 \le
 1+\frac{\log K_G^{\K}}{\tau_N}.
\]
Since $\beta_{d_N}^{\K}\to1$, both error terms tend to zero, proving (iii).

Finally assume $\tau_N\to\tau<\infty$. Again $p_N\to2$, so Lemma \ref{lem:fixed-Hilbert-extremizers} gives the lower bound $K_G^{\K}$. The lower half of Theorem \ref{thm:A} gives
\[
 \G_{p_N}^{\K}(r_N;m_N,n_N)
 \ge B_{d_N}^{\K}e^{\tau_N},
\]
and therefore
\[
 \liminf_N\G_{p_N}^{\K}(r_N;m_N,n_N)
 \ge e^\tau\underline B_{\K}.
\]
The upper bound in \eqref{eq:tau-master-rect} gives the last inequality in \eqref{eq:critical-envelope}. Theorem \ref{thm:B} gives $\underline B_{\K}\ge1$.

The critical scale \eqref{eq:critical-scale} is the boundedness threshold from (i). If $p_N\to2$, then
\[
 \frac1{p_N}-\frac12
 =-\frac{p_N-2}{2p_N}
 =-\frac{p_N-2}{4}+o(|p_N-2|),
\]
which gives the equivalent formulation in terms of $|p_N-2|$.
\end{proof}

\section{The complex diagonal profile}

Theorem~\ref{thm:C} uses complex interpolation between the exact
rectangular endpoints and the Hilbertian problem. In the square case
one can retain considerably more information than a single interpolated
upper bound: the entire passage through the critical window can be
encoded by a convex profile.

We therefore restrict in this section to $\K=\C$ and to square
parameters. For a fixed matrix $A\in\C^{d\times d}$, Lemma
\ref{lem:operator-realization} realizes
\[
G_p^{(d)}(A)
\]
as the norm of
\[
T_A^{(p)}:
\ell_\infty^d(\ell_p^d)
\longrightarrow
\ell_1^d(\ell_p^d).
\]
Consequently, complex interpolation makes
\[
\frac1p\longmapsto\log G_p^{(d)}(A)
\]
convex. The corresponding multiplicative estimates remain valid after
division by $\|A\|_{\infty,1}$ and passage to the supremum over $A$.

For $d\ge2$ put
\[
L_d=\frac12\log d
\]
and, for $|s|\le L_d$, define
\[
\frac1{p_d(s)}
=
\frac12-\frac{s}{\log d},
\qquad
F_d(s)
=
\log\Gamma_{p_d(s)}^\C(d;d,d).
\]
Thus $s=0$ is the Hilbertian point and
$s=L_d$ is the endpoint $p=\infty$.
The next proposition describes the additional shape information
contained in this one-dimensional profile.

\begin{proof}[Proof of Proposition \ref{prop:complex-profile}]
Since $p_d(-s)=p_d(s)'$, Lemma \ref{lem:duality} gives $F_d(-s)=F_d(s)$. Complex interpolation shows that $F_d$ is convex because $s\mapsto1/p_d(s)$ is affine.

For $1\le p,q\le\infty$, direct comparison of the $\ell_p^d$ norms gives
\begin{equation}\label{eq:dimension-comparison-diagonal}
 d^{-|1/p-1/q|}\G_q^{\C}(d;d,d)
 \le
 \G_p^{\C}(d;d,d)
 \le
 d^{|1/p-1/q|}\G_q^{\C}(d;d,d).
\end{equation}
Applying this to $p=p_d(s)$ and $q=p_d(t)$ yields
\[
 |F_d(s)-F_d(t)|\le|s-t|,
\]
so $F_d$ is $1$-Lipschitz.

For $0\le s\le L_d$,
\[
 H_d(s)
 =F_d(s)-s
 =\log\left(d^{-s/\log d}
 \G_{p_d(s)}^{\C}(d;d,d)\right).
\]
If $0\le s\le t\le L_d$, the $1$-Lipschitz estimate gives
\[
 F_d(t)-F_d(s)\le t-s.
\]
After subtracting $t-s$ this becomes $H_d(t)\le H_d(s)$, so $H_d$ is nonincreasing. Since $H_d(s)=F_d(s)-s$, convexity of $F_d$ also gives convexity of $H_d$. At $s=0$, $p_d(0)=2$; at $s=L_d$, $p_d(L_d)=\infty$. Hence
\[
 H_d(0)=\log\G_2^{\C}(d;d,d),
 \qquad
 H_d(L_d)=\log B_d^{\C}.
\]
Monotonicity gives
\[
 F_d(s)\ge s+\log B_d^{\C},
\]
and even convexity gives $F_d(s)\ge F_d(0)$. This proves \eqref{eq:profile-lower}. The chord bound for the convex function $H_d$ gives
\[
 H_d(s)
 \le
 \left(1-\frac{2s}{\log d}\right)H_d(0)
 +\frac{2s}{\log d}H_d(L_d),
\]
which is \eqref{eq:profile-upper} after adding $s$ and using evenness.

The functions $F_d$ are uniformly $1$-Lipschitz on every fixed compact interval once $d$ is large enough. Since $F_d(0)=\log\G_2^{\C}(d;d,d)\le\log K_G^{\C}$, they are also uniformly bounded on each fixed compact interval. Moreover,
\[
 \G_2^{\C}(d;d,d)\longrightarrow K_G^{\C}.
\]
Indeed, the sequence is nondecreasing by padding, bounded above by $K_G^{\C}$, and every fixed finite Grothendieck configuration embeds in a sufficiently large square problem. Proposition \ref{prop:endpoint} bounds $B_d^{\C}$ in $[1,2/\sqrt\pi]$. A diagonal Arzel\`a--Ascoli argument therefore gives the stated subsequential convergence. Passing to the limit in \eqref{eq:profile-lower} and \eqref{eq:profile-upper} proves \eqref{eq:profile-limit}.
\end{proof}

\section{Relation with Marcinkiewicz--Zygmund theory and prior work}

The exponent $|1/p-1/2|$ already appears in the finite-dimensional
Marcinkiewicz--Zygmund theory of Defant and Junge
\cite{DefantJunge1997}. In the $L_\infty\to L_1$ setting, their general
operator estimates give the growth rate $N^{|1/p-1/2|}$ outside the
dimension-free range. For the matrix multiplication tensor, Friedland,
Lim and Zhang \cite{FriedlandLimZhang2019} identified Grothendieck's
constant with $\norm{\mu_{r,m,n}}_{1,2,\infty}$ and obtained
\[
 \norm{\mu_{r,m,n}}_{1,p,\infty}
 \le K_G^{\K}r^{|1/p-1/2|}.
\]
Thus the exponent and the special role of $p=2$ are already present in
these earlier results.

Theorems \ref{thm:A} and \ref{thm:B} refine this estimate for the
rectangular multiplication tensor by replacing the vector dimension
$r$ with the effective dimension $d=\min\{r,m,n\}$, with matching
upper and lower estimates. The pointwise identity
$G_2^{(r)}(A)=G_2^{(d)}(A)$ explains this reduction in the Hilbertian
case, while Remark \ref{rem:rectangular-endpoint-saturation} shows
that the corresponding pointwise stabilization need not hold at
$p=\infty$.

Theorem \ref{thm:C} keeps information which is lost when one passes
immediately to the single parameter $d$. The exact endpoint
compression produces two rectangular lower envelopes, both valid for
the full H\"older scale and involving all intermediate dimensions.
Over the complex field, interpolation between the Hilbertian problem
and the exact rectangular endpoints gives the corresponding
endpoint-sensitive upper envelopes. Thus the effective dimension
controls the universal power law, whereas the tensor ratios
$\mathcal R_{\K}(m,k)$ and $\mathcal R_{\K}(n,k)$ retain finer
aspect-ratio information.

For dimension-dependent exponents, Theorem \ref{thm:D} identifies the
boundedness threshold and describes the subcritical, critical, and
supercritical regimes through the parameter
$\tau_N=|1/p_N-1/2|\log d_N$.

\section{Endpoint questions}

The remaining endpoint quantity
\[
 B_d^{\K}=
 \frac1{\sqrt d}\rho(\ell_1^d(\K),\ell_1^d(\K))
\]
controls the lower constant in Theorem \ref{thm:A}. Theorem \ref{thm:B} and Proposition \ref{prop:endpoint} give the quantitative real bounds
\[
 \sqrt{\frac{d}{h_+(d)}}
 \le B_d^{\R}
 \le
 \frac{\sqrt d}{M_d},
\]
with
\[
 1-\sqrt{\frac{d}{h_+(d)}}=O(d^{-5/6}),
 \qquad
 \limsup_{d\to\infty}B_d^{\R}\le\sqrt{\frac\pi2},
\]
and the dimensionwise complex bounds
\[
 1\le B_d^{\C}\le\frac2{\sqrt\pi}.
\]
Determining the asymptotic behavior of $B_d^{\K}$ would sharpen the lower edge of the critical window in Theorem \ref{thm:D}. In the real case this is equivalently the asymptotic problem for the restricted summing norm
\[
 \frac1{\sqrt d}\pi_1^{(d)}(I_{\ell_1^d}).
\]

\section*{Declarations}

\noindent\textit{Funding.}
J. Santos and D. Serrano-Rodr\'iguez  are supported by Grants No. 406457/2023-9 (CNPq/MCTI No. 10/2023) and No. 403964/2024-5 (MCTI/CNPq No. 16/2024), both from the Conselho Nacional de Desenvolvimento Científico e Tecnológico (CNPq, Brazil). In addition, J. Santos is supported by Grant No. 305655/2025-6.

\medskip

\noindent\textit{Competing interests.}
The authors declare no competing interests.

\medskip

\noindent\textit{Data availability.}
No datasets were generated or analysed in this work.

\medskip

\noindent\textit{Use of artificial intelligence.}
Generative artificial-intelligence tools, including OpenAI's ChatGPT, were used for
auxiliary verification of calculations, bibliographic consistency, and language and
presentation assistance during manuscript preparation. All mathematical arguments and
final expository decisions were independently checked and remain the responsibility of
the authors.

\end{document}